\documentclass[a4paper,12pt]{article}

\usepackage[utf8]{inputenc}
\usepackage[english]{babel}
\usepackage{amsmath, amssymb, amsthm}
\usepackage{enumerate}
\usepackage[colorlinks=true,linkcolor=blue,citecolor=blue]{hyperref}

\title{An analogue of Mahler's transference theorem for multiplicative Diophantine approximation}
\author{Oleg\,N.\,German}
\date{}

\theoremstyle{definition}
\newtheorem{definition}{Definition}

\newtheorem*{notation*}{Notation}

\theoremstyle{remark}

\newtheorem*{remark*}{Remark}

\theoremstyle{plain}
\newtheorem{theorem}{Theorem}

\newtheorem*{question*}{Question}
\newtheorem*{statement*}{Statement}
\newtheorem*{corollary*}{Corollary}

\renewcommand{\phi}{\varphi}

\renewcommand{\vec}[1]{\mathbf{#1}}
\renewcommand{\geq}{\geqslant}
\renewcommand{\leq}{\leqslant}

\newcommand{\R}{\mathbb{R}}
\newcommand{\Z}{\mathbb{Z}}

\newcommand{\La}{\Lambda}

\newcommand{\cF}{\mathcal{F}}
\newcommand{\cG}{\mathcal{G}}

\newcommand{\cP}{\mathcal{P}}
\newcommand{\cQ}{\mathcal{Q}}

\newcommand{\tr}[1]{{#1}^\intercal}

\begin{document}

\maketitle

\begin{abstract}
  \noindent
  Khintchine's and Dyson's transference theorems can be very easily deduced from Mahler's transference theorem. In the multiplicative setting an obstacle appears, which does not allow deducing the multiplicative transference theorem immediately from Mahler's theorem. Some extra considerations are required, for instance, induction by the dimension. In this paper we propose an analogue of Mahler's theorem which implies the multiplicative transference theorem immediately.
\end{abstract}


\section{Introduction}\label{sec:intro}

Consider a matrix
\[
  \Theta=
  \begin{pmatrix}
    \theta_{11} & \cdots & \theta_{1m} \\
    \vdots & \ddots & \vdots \\
    \theta_{n1} & \cdots & \theta_{nm}
  \end{pmatrix},\qquad
  \theta_{ij}\in\R,\quad m+n\geq3,
\]
and a system of linear equations
\begin{equation*} 
  \Theta\vec x=\vec y
\end{equation*}
with variables $\vec x=(x_1,\ldots,x_m)\in\R^m$, $\vec y=(y_1,\ldots,y_n)\in\R^n$. One of the main questions in the theory od Diophantine approximation is how small the vector $\Theta\vec x-\vec y$ can be as $\vec x$ and $\vec y$ range independently through $\Z^m\backslash\{\vec 0\}$ and $\Z^n$ respectively. There are several classical ways of measuring the ``size'' of a vector. One can choose a norm, for instance, the sup-norm, one can alter it turning it into a so called weighted norm, or one can consider the product of the absolute values of a vector's coordinates. In each of those settings there exist \emph{transference theorems} --- statements reflecting the relation between the approximation properties of $\Theta$ and those of $\tr\Theta$. They are usually formulated in terms of \emph{Diophantine exponents}, which are probably the simplest quantities responsible for the approximation properties.

Given a positive integer $k$ and $\vec z=(z_1,\ldots,z_k)\in\R^k$, we denote
\[
  |\vec z|=\max_{1\leq i\leq k}|z_i|,
  \qquad
  \Pi(\vec z)=\prod_{\begin{subarray}{c}1\leq i\leq k\end{subarray}}|z_i|^{1/k},
  \qquad
  \Pi'(\vec z)=\prod_{1\leq i\leq k}\max\big(1,|z_i|\big)^{1/k}.
\]

\begin{definition} \label{def:beta}
  Supremum of real numbers $\gamma$ for which there exists $t$ however large such that the system of inequalities
  \begin{equation} \label{eq:beta}
    |\vec x|\leq t,\qquad|\Theta\vec x-\vec y|\leq t^{-\gamma}
  \end{equation}
  admits a solution $(\vec x,\vec y)\in\Z^m\oplus\Z^n$ with nonzero $\vec x$ is called the \emph{Diophantine exponent} of $\Theta$ and is denoted by $\omega(\Theta)$.
\end{definition}

\begin{definition} \label{def:mbeta}
  Supremum of real numbers $\gamma$ for which there exists $t$ however large such that the system of inequalities
  \begin{equation} \label{eq:mbeta}
    \Pi'(\vec x)\leq t,
    \qquad
    \Pi(\Theta\vec x-\vec y)\leq t^{-\gamma}
  \end{equation}
  admits a solution $(\vec x,\vec y)\in\Z^m\oplus\Z^n$ with nonzero $\vec x$ is called the \emph{multiplicative Diophantine exponent} of $\Theta$ and is denoted by $\omega_\times(\Theta)$.
\end{definition}

For each $\vec z\in\R^k$ we have
\[
  \Pi(\vec z)\leq|\vec z|,
\]
and for each $\vec z\in\Z^k$ we have
\[
  |\vec z|^{1/k}\leq\Pi'(\vec z)\leq|\vec z|.
\]
Hence
\begin{equation}\label{eq:trivial_inequalities}
  m/n\leq
  \omega(\Theta)\leq
  \omega_\times(\Theta)\leq
  \begin{cases}
    m\omega(\Theta)\quad
    \text{ for }n=1 \\
    +\infty\qquad\
    \text{ for }n\geq2
  \end{cases},
\end{equation}
where the first inequality is a consequence of Minkowski's convex body theorem.

The inequalities \eqref{eq:trivial_inequalities} can be called trivial. The transference theorems mentioned above provide the following nontrivial relations:
\begin{equation} \label{eq:dyson_transference}
  \omega(\tr\Theta)\geq\frac{n\omega(\Theta)+n-1}{(m-1)\omega(\Theta)+m}
\end{equation}
and
\begin{equation}\label{eq:german_multiplicative_dysonlike}
  \omega_\times(\tr\Theta)\geq\frac{n\omega_\times(\Theta)+n-1}{(m-1)\omega_\times(\Theta)+m}\,.
\end{equation}
The inequality \eqref{eq:dyson_transference} belongs to Dyson \cite{dyson}, the inequality \eqref{eq:german_multiplicative_dysonlike} was proved by the author in \cite{german_mult}. One can notice that the inequalities look identical, however, there is an essential difference between their proofs. Dyson's inequality follows almost immediately from Mahler's transference theorem (see \cite{cassels_DA}, \cite{schmidt_DA}, and also \cite{german_MJCNT_2012}, \cite{german_evdokimov}), whereas the inequality for the multiplicative exponents, along with Mahler's theorem, requires induction by $n$. Roughly speaking, the reason is that the functionals $\Pi(\cdot)$ and $\Pi'(\cdot)$ are not the same.

The purpose of this paper is to find an analogue of Mahler's transference theorem so that it would imply \eqref{eq:german_multiplicative_dysonlike} as immediately as the classical Mahler theorem implies \eqref{eq:dyson_transference}.

The rest of the paper is organised as follows. In Section \ref{sec:mahler_vs_dyson} we formulate Mahler's theorem and show how to derive Dyson's inequality from it. In Section \ref{sec:multiplicative_analogue_of_mahler} we formulate and prove the main result of this paper. In Section \ref{sec:proof_of_multiplicative_transference} we derive \eqref{eq:german_multiplicative_dysonlike} from our result.

\section{Mahler's theorem and Dyson's inequality}\label{sec:mahler_vs_dyson}

Set
\[
  d=m+n.
\]
In his original paper \cite{mahler_casopis_linear}, Mahler formulated his famous theorem in terms of bilinear forms with integer coefficients (see also \cite{cassels_DA} and \cite{german_evdokimov})). In \cite{german_evdokimov} Mahler's theorem is interpreted in terms of pseudocompound parallelepipeds and dual lattices. We deem this interpretation more apt for applications. A pseudocompound parallelepiped is a concept proposed in Schmidt's book \cite{schmidt_DA}, it is a simplification of what Mahler calls in his papers \cite{mahler_compound_bodies_I}, \cite{mahler_compound_bodies_II} the $(d-1)$-th compound body of a parallelepiped.

\begin{definition}\label{def:pseudo_compound}
  Let $\eta_1,\ldots,\eta_d$ be positive real numbers. Consider the parallelepiped
  \begin{equation}\label{eq:pseudo_compound}
    \cP=\Big\{ \vec z=(z_1,\ldots,z_d)\in\R^d \,\Big|\, |z_i|\leq\eta_i,\ i=1,\ldots,d \Big\}.
  \end{equation}
  The parallelepiped
  \[
    \cP^\ast=\Big\{ \vec z=(z_1,\ldots,z_d)\in\R^d \,\Big|\, |z_i|\leq\frac1{\eta_i}\prod_{j=1}^d\eta_j,\ i=1,\ldots,d \Big\}
  \]
  is called the \emph{pseudocompound} of $\cP$.
\end{definition}

We remind that, given a full-rank lattice $\La$ in $\R^d$, its dual lattice $\La^\ast$ is defined as
\[
  \La^\ast=\big\{\, \vec z\in\R^d \,\big|\ \langle\vec z,\vec z'\rangle\in\Z\text{ for each }\vec z'\in\La \,\big\},
\]
where $\langle \,\cdot\,,\cdot\,\rangle$ denotes the inner product.

The following version of Mahler's transference theorem is proposed in \cite{german_evdokimov}.

\begin{theorem}\label{t:mahler_reformulated}
  Let $\La$ be a full-rank lattice in $\R^d$ with determinant equal to $1$. Let $\cP$ be a parallelepiped centered at the origin with faces parallel to the coordinate planes. Then
  \begin{equation*}
    \cP^\ast\cap\La^\ast\neq\{\vec 0\}
    \implies
    c\cP\cap\La\neq\{\vec 0\}
  \end{equation*}
  with $c=\big(\sqrt d\big)^{1/(d-1)}$.
\end{theorem}

Theorem \ref{t:mahler_reformulated} is actually a strengthening of the original Mahler's theorem. Mahler formulated his theorem with $d-1$ instead of $c$. We note however that, from the point of view of Diophantine exponents, any constant (depending on $d$ only) will do.

Let us show how to derive \eqref{eq:dyson_transference} from Theorem \ref{t:mahler_reformulated}. Recall that $d=m+n$.

Consider the lattices
\begin{equation}\label{eq:lattices_arbitrary_nm}
  \La=\La(\Theta)=
  \begin{pmatrix}
    \vec I_m & \\
    -\Theta  & \vec I_n
  \end{pmatrix}
  \Z^d,
  \qquad
  \La^\ast=\La^\ast(\Theta)=
  \begin{pmatrix}
    \vec I_m & \tr\Theta \\
    & \vec I_n
  \end{pmatrix}
  \Z^d.
\end{equation}
Clearly, $\La^\ast$ is the dual lattice of $\La$. Furthermore, for each set of positive $t$, $\gamma$, $s$, $\delta$, let us define the parallelepipeds
\begin{align}
  \label{eq:t_gamma_family_arbitrary_nm}
  \cP(t,\gamma) & =\Bigg\{\,\vec z=(z_1,\ldots,z_d)\in\R^d \ \Bigg|
                        \begin{array}{l}
                          |z_j|\leq t,\qquad\ \ j=1,\ldots,m \\
                          |z_{m+i}|\leq t^{-\gamma},\ \ i=1,\ldots,n
                        \end{array} \Bigg\}, \\
  \label{eq:s_delta_family_arbitrary_nm}
  \cQ(s,\delta) & =\Bigg\{\,\vec z=(z_1,\ldots,z_d)\in\R^d \ \Bigg|
                        \begin{array}{l}
                          |z_j|\leq s^{-\delta},\quad\ \ j=1,\ldots,m \\
                          |z_{m+i}|\leq s,\quad\ \ i=1,\ldots,n
                        \end{array} \Bigg\}.
\end{align}
Then
\begin{equation}\label{eq:omega_vs_parallelepipeds_arbitrary_nm}
\begin{aligned}
  \omega(\Theta) & =\sup\bigg\{ \gamma\geq\frac mn \,\bigg|\ \forall\,t_0\in\R\,\ \exists\,t>t_0:\ \cP(t,\gamma)\cap\La\neq\{\vec 0\} \bigg\}, \\
  \omega(\tr\Theta) & =\sup\bigg\{ \delta\geq\frac nm\ \bigg|\ \forall\,s_0\in\R\,\ \exists\,s>s_0:\ \cQ(s,\delta)\cap\La^\ast\neq\{\vec 0\} \bigg\}.
\end{aligned}
\end{equation}
If $t$, $\gamma$, $s$, $\delta$ are related by
\begin{equation}\label{eq:Q_is_P_star_arbitrary_nm}
  t=s^{((n-1)\delta+n)/(d-1)},
  \qquad
  \gamma=\frac{m\delta+m-1}{(n-1)\delta+n},
\end{equation}
then $\cQ(s,\delta)$ is the pseudocompound of $\cP(t,\gamma)$, that is $\cQ(s,\delta)=\cP(t,\gamma)^\ast$. By Theorem \ref{t:mahler_reformulated} we get
\begin{equation*}
  \cQ(s,\delta)\cap\La^\ast\neq\{\vec 0\}
  \implies
  c\cP(t,\gamma)\cap\La\neq\{\vec 0\}.
\end{equation*}
Hence, in view of \eqref{eq:omega_vs_parallelepipeds_arbitrary_nm},
\[
  \omega(\tr\Theta)\geq\delta
  \implies
  \omega(\Theta)\geq\gamma=\frac{m\delta+m-1}{(n-1)\delta+n}\,.
\]
Thus,
\begin{equation*}
  \omega(\Theta)\geq\frac{\omega(m\tr\Theta)+m-1}{(n-1)\omega(\tr\Theta)+n}\,.
\end{equation*}
Swapping the triple $(\Theta,m,n)$ for $(\tr\Theta,n,m)$, we get \eqref{eq:dyson_transference}.

\section{An analogue of Mahler's theorem}\label{sec:multiplicative_analogue_of_mahler}

For each tuple $(\pmb\lambda,\pmb\mu)=(\lambda_1,\ldots,\lambda_m,\mu_1,\ldots,\mu_n)\in\R_+^d$, we define the parallelepiped $\cP(\pmb\lambda,\pmb\mu)$ as
\begin{equation}\label{eq:prallelepipeds_mult}
  \cP(\pmb\lambda,\pmb\mu)=\Bigg\{\,\vec z=(z_1,\ldots,z_d)\in\R^d \ \Bigg|
  \begin{array}{l}
    |z_j|\leq\lambda_j,\quad\,\ j=1,\ldots,m \\
    |z_{m+i}|\leq\mu_i,\ \ i=1,\ldots,n
  \end{array} \Bigg\}.
\end{equation}
Let us also set
\begin{equation}\label{eq:lambda_mu_ast}
\begin{aligned}
  & \lambda_j^\ast=\lambda_j^{-1}\prod_{k=1}^m\lambda_k\prod_{k=1}^n\mu_k,
    \qquad j=1,\ldots,m,
    \vphantom{\bigg|} \\
  & \mu_i^\ast=\mu_i^{-1}\prod_{k=1}^m\lambda_k\prod_{k=1}^n\mu_k,
    \qquad\,i=1,\ldots,n.
\end{aligned}
\end{equation}
Then, clearly, $\cP(\pmb\lambda,\pmb\mu)^\ast=\cP(\pmb\lambda^\ast,\pmb\mu^\ast)$. Finally, we define the tuple $\hat{\pmb\lambda}=(\hat\lambda_1,\ldots,\hat\lambda_m)$ as follows. Let us sort the elements of $\pmb\lambda$ in ascending order: $\lambda_{j_1}\leq\ldots\leq\lambda_{j_m}$. If $\lambda_{j_1}\geq1$, we set $\hat{\pmb\lambda}=\pmb\lambda$. If $\lambda_{j_1}<1$, we set $p$ to be the greatest index such that $\lambda_{j_1}\cdot\ldots\cdot\lambda_{j_p}<1$ and define $\hat\lambda_1,\ldots,\hat\lambda_m$ as
\begin{equation}\label{eq:lambda_mu_hat_general_case}
\begin{aligned}
  & \hat\lambda_{j_i}=1,
    \qquad\qquad\qquad\qquad\qquad\qquad\ i=1,\ldots,p, \\
  & \hat\lambda_{j_i}=\lambda_{j_i}\big(\lambda_{j_1}\cdot\ldots\cdot\lambda_{j_k}\big)^{1/(m-p)},
    \qquad\ \ i=p+1,\ldots,m.
    \vphantom{1^{\big|}}
\end{aligned}
\end{equation}

The following theorem is the main result of the paper.

\begin{theorem}\label{t:multi_transference_essence}
  Let $\La$ and $\La^\ast$ be defined by \eqref{eq:lattices_arbitrary_nm}. Consider arbitrary tuples $\pmb\lambda=(\lambda_1,\ldots,\lambda_m)\in\R_+^m$ and $\pmb\mu=(\mu_1,\ldots,\mu_n)\in\R_+^n$. Suppose
  \begin{equation}\label{eq:multi_transference_essence_Pi_geq_1}
    \Pi(\pmb\lambda)\geq1.
  \end{equation}
   Let $\pmb\lambda^\ast$, $\pmb\mu^\ast$ be defined by \eqref{eq:lambda_mu_ast}, and let $\hat{\pmb\lambda}$ be defined by \eqref{eq:lambda_mu_hat_general_case}. Then
  \begin{equation}\label{eq:multi_transference_essence_Pi_hat_equals_Pi}
    \min_{1\leq j\leq m}\hat\lambda_j\geq1,
    \qquad
    \Pi'(\hat{\pmb\lambda})=
    \Pi(\hat{\pmb\lambda})=
    \Pi(\pmb\lambda),
  \end{equation}
  and
  \begin{equation}\label{eq:multi_transference_essence}
    \cP(\pmb\lambda^\ast,\pmb\mu^\ast)\cap\La^\ast\neq\{\vec 0\}
    \implies
    c_1\cP(\hat{\pmb\lambda},\pmb\mu)\cap\La\neq\{\vec 0\}
  \end{equation}
  with $c_1=\big(\sqrt{n+1}\big)^{1/n}$.
\end{theorem}

\begin{proof}
  Without loss of generality, we may assume that
  \[
    \lambda_1\leq\ldots\leq\lambda_m.
  \]
  If $\lambda_1\geq1$, then $\hat{\pmb\lambda}=\pmb\lambda$, \eqref{eq:multi_transference_essence_Pi_hat_equals_Pi} is obvious, and \eqref{eq:multi_transference_essence} is provided by Theorem \ref{t:mahler_reformulated}, since $c\leq c_1$. Let us assume that $\lambda_1<1$. Then $p$ is correctly defined and $p<m$, since \eqref{eq:multi_transference_essence_Pi_geq_1} holds. Hence \eqref{eq:multi_transference_essence_Pi_hat_equals_Pi} follows immediately.

  Let us consider the truncated tuples
  \[
    \pmb\lambda_\downarrow=(\lambda_{p+1},\ldots,\lambda_{m}),
    \qquad
    \pmb\lambda^\ast_\downarrow=(\lambda^\ast_{p+1},\ldots,\lambda^\ast_{m}),
    \qquad
    \hat{\pmb\lambda}_\downarrow=(\hat\lambda_{p+1},\ldots,\hat\lambda_m).
  \]
  Then
  \[
    \cP(\hat{\pmb\lambda}_\downarrow,\pmb\mu)^\ast=
    \Bigg\{\,(z_{p+1},\ldots,z_d)\in\R^{d-p} \ \Bigg|
    \begin{array}{l}
      |z_j|\leq \varkappa\lambda_j^\ast,\quad\, j=p+1,\ldots,m \\
      |z_{m+i}|\leq\mu_i^\ast,\quad i=1,\ldots,n
    \end{array} \Bigg\},
  \]
  where $\varkappa=\big(\lambda_1\cdot\ldots\cdot\lambda_p\big)^{-1/(m-p)}$. Since $\varkappa>1$, we have
  \begin{equation}\label{eq:ammendment_of_bad}
    \cP(\pmb\lambda_\downarrow^\ast,\pmb\mu^\ast)
    \subset
    \cP(\hat{\pmb\lambda}_\downarrow,\pmb\mu)^\ast.
  \end{equation}
  Let us also consider the matrix
  \[
    \Theta_\downarrow=
    \begin{pmatrix}
      \theta_{1\,p+1} & \cdots & \theta_{1m} \\
      \vdots & \ddots & \vdots \\
      \theta_{n\,p+1} & \cdots & \theta_{nm}
    \end{pmatrix}
  \]
  obtained from $\Theta$ by deleting the first $p$ columns, and the lattices
  \begin{equation*}
    \La_\downarrow=
    \begin{pmatrix}
      \vec I_{m-p}        & \\
      -\Theta_\downarrow  & \vec I_n
    \end{pmatrix}
    \Z^{d-p},
    \qquad
    \La^\ast_\downarrow=
    \begin{pmatrix}
      \vec I_{m-p} & \tr\Theta_\downarrow \\
                   & \vec I_n
    \end{pmatrix}
    \Z^{d-p}.
  \end{equation*}
  We make the following two crucial observations: first, the set
  \[
    \Big\{ (0,\ldots,0,z_{p+1},\ldots,z_d)\in\R^d \ \Big|\
                      (z_{p+1},\ldots,z_d)\in\La_\downarrow \Big\}
  \]
  is a sublattice of $\La$; second, the set
  \[
    \Big\{ (0,\ldots,0,z_{p+1},\ldots,z_d)\in\R^d \ \Big|\
                      (z_{p+1},\ldots,z_d)\in\La^\ast_\downarrow \Big\}
  \]
  is the orthogonal projection of $\La^\ast$ onto the $(z_{p+1},\ldots,z_d)$--coordinate plane. Hence
  \begin{equation}\label{eq:down_and_up}
  \begin{aligned}
    \cP(\pmb\lambda^\ast,\pmb\mu^\ast)\cap\La^\ast\neq\{\vec 0\}
    & \implies
    \cP(\pmb\lambda^\ast_\downarrow,\pmb\mu^\ast)\cap\La^\ast_\downarrow\neq\{\vec 0\}, \\
    \cP(\hat{\pmb\lambda}_\downarrow,\pmb\mu)\cap\La_\downarrow\neq\{\vec 0\}
    & \implies
    \cP(\hat{\pmb\lambda},\pmb\mu)\cap\La\neq\{\vec 0\}.
    \vphantom{1^{\big|}}
  \end{aligned}
  \end{equation}
  Finally, by Theorem \ref{t:mahler_reformulated} we have
  \begin{equation}\label{eq:mahler_downarrowed}
    \cP(\hat{\pmb\lambda}_\downarrow,\pmb\mu)^\ast\cap\La^\ast_\downarrow\neq\{\vec 0\}
    \implies
    c_2\cP(\hat{\pmb\lambda}_\downarrow,\pmb\mu)\cap\La_\downarrow\neq\{\vec 0\}
  \end{equation}
  with $c_2=\big(\sqrt{d-p}\big)^{1/(d-p-1)}$. Gathering up together \eqref{eq:ammendment_of_bad}, \eqref{eq:down_and_up}, \eqref{eq:mahler_downarrowed}, and taking into account that $c_2\leq c_1$, we get the following chain of implications:
  \begin{multline*}
    \cP(\pmb\lambda^\ast,\pmb\mu^\ast)\cap\La^\ast\neq\{\vec 0\}
    \implies
    \cP(\pmb\lambda^\ast_\downarrow,\pmb\mu^\ast)\cap\La^\ast_\downarrow\neq\{\vec 0\}
    \implies \\ \implies
    \cP(\hat{\pmb\lambda}_\downarrow,\pmb\mu)^\ast\cap\La^\ast_\downarrow\neq\{\vec 0\}
    \implies
    c_2\cP(\hat{\pmb\lambda}_\downarrow,\pmb\mu)\cap\La_\downarrow\neq\{\vec 0\}
    \implies \vphantom{\bigg|} \\ \implies
    c_2\cP(\hat{\pmb\lambda},\pmb\mu)\cap\La\neq\{\vec 0\}
    \implies
    c_1\cP(\hat{\pmb\lambda},\pmb\mu)\cap\La\neq\{\vec 0\},
  \end{multline*}
  which proves \eqref{eq:multi_transference_essence}.
\end{proof}

\section{Proof of the multiplicative transference inequality}\label{sec:proof_of_multiplicative_transference}

Let us show how to derive \eqref{eq:german_multiplicative_dysonlike} from Theorem \ref{t:multi_transference_essence}.
For every positive $t$, $\gamma$, $s$, $\delta$, let us define the following two families of parallelepipeds:
\begin{align*}
  \cF(t,\gamma) & =\Big\{\, \cP(\pmb\lambda,\pmb\mu) \ \Big|\
                            \Pi(\pmb\lambda)=t,\
                            \Pi(\pmb\mu)=t^{-\gamma},\
                            \min_{1\leq j\leq m}\lambda_j\geq1 \Big\}, \\
  \cG(s,\delta) & =\Big\{\, \cP(\pmb\lambda,\pmb\mu) \ \Big|\
                            \Pi(\pmb\lambda)=s^{-\delta},\
                            \Pi(\pmb\mu)=s,\
                            \min_{1\leq i\leq n}\mu_i\geq1 \Big\}.
\end{align*}
Each parallelepiped $\cP(\pmb\lambda,\pmb\mu)$ satisfying the conditions
\begin{equation}\label{eq:lambda_mu_for_omega_mult}
  \Pi'(\pmb\lambda)\leq t,
  \qquad
  \Pi(\pmb\mu)\leq t^{-\gamma}
\end{equation}
is contained in a parallelepiped from $\cF(t,\gamma)$. Conversely, each parallelepiped $\cP(\pmb\lambda,\pmb\mu)$ from $\cF(t,\gamma)$ satisfies \eqref{eq:lambda_mu_for_omega_mult}. Similarly, each parallelepiped $\cP(\pmb\lambda,\pmb\mu)$ satisfying the conditions
\begin{equation}\label{eq:lambda_mu_for_omega_mult_transpose}
  \Pi(\pmb\lambda)\leq s^{-\delta},
  \qquad
  \Pi'(\pmb\mu)\leq s
\end{equation}
is contained in a parallelepiped from $\cG(s,\delta)$. And conversely, each parallelepiped $\cP(\pmb\lambda,\pmb\mu)$ from $\cG(s,\delta)$ satisfies \eqref{eq:lambda_mu_for_omega_mult_transpose}. Thus, for multiplicative exponents, the following analogue of \eqref{eq:omega_vs_parallelepipeds_arbitrary_nm} holds:
\begin{equation}\label{eq:omega_mult_vs_parallelepipeds}
\begin{aligned}
  \omega_\times(\Theta) & =
  \sup\bigg\{ \gamma\geq\frac mn \,\bigg|\
              \forall\,t_0\in\R\,\ \exists\,t>t_0:\,\
              \exists\cP\in\cF(t,\gamma):\ \cP\cap\La\neq\{\vec 0\} \bigg\}, \\
  \omega_\times(\tr\Theta) & =
  \sup\bigg\{ \delta\geq\frac nm\ \bigg|\
              \forall\,s_0\in\R\,\ \exists\,s>s_0:\,\
              \exists\cP\in\cG(s,\delta):\ \cP\cap\La^\ast\neq\{\vec 0\} \bigg\}.
\end{aligned}
\end{equation}
Let us assume again that $t$, $\gamma$, $s$, $\delta$ are related by \eqref{eq:Q_is_P_star_arbitrary_nm}. Consider an arbitrary parallelepiped $\cP(\pmb\lambda,\pmb\mu)$ such that $\cP(\pmb\lambda^\ast,\pmb\mu^\ast)\in\cG(s,\delta)$. Then
\[
  \Pi(\pmb\lambda)=t,\qquad\Pi(\pmb\mu)=t^{-\gamma}.
\]
We cannot guarantee that $\pmb\lambda$ has no components strictly less that $1$, so generally it is not true that $\cP(\pmb\lambda,\pmb\mu)\in\cF(t,\gamma)$. Nevertheless, if $t\geq1$, then by Theorem \ref{t:multi_transference_essence} we do have $\cP(\hat{\pmb\lambda},\pmb\mu)\in\cF(t,\gamma)$, and moreover,
\begin{equation*}
  \cP(\pmb\lambda^\ast,\pmb\mu^\ast)\cap\La^\ast\neq\{\vec 0\}
  \implies
  c_1\cP(\hat{\pmb\lambda},\pmb\mu)\cap\La\neq\{\vec 0\}.
\end{equation*}
Hence, in view of \eqref{eq:omega_mult_vs_parallelepipeds},
\[
  \omega_\times(\tr\Theta)\geq\delta
  \implies
  \omega_\times(\Theta)\geq\gamma=\frac{m\delta+m-1}{(n-1)\delta+n}\,.
\]
Thus,
\begin{equation*}
  \omega_\times(\Theta)\geq\frac{\omega_\times(m\tr\Theta)+m-1}{(n-1)\omega_\times(\tr\Theta)+n}\,.
\end{equation*}
Swapping the triple $(\Theta,m,n)$ for $(\tr\Theta,n,m)$, we get \eqref{eq:german_multiplicative_dysonlike}.

\paragraph{Acknowledgements.}


The author is a winner of the ``Junior Leader'' contest conducted by Theoretical Physics and Mathematics Advancement Foundation ``BASIS'' and would like to thank its sponsors and jury.



\begin{thebibliography}{99}

\bibitem
    {dyson}
    \textsc{F.\,J.\,Dyson}
    \textit{On simultaneous Diophantine approximations.}
    Proc. London Math. Soc., (2) 49 (1947), 409--420.
\bibitem
    {german_mult}
    \textsc{O.\,N.\,German}
    \textit{Transference  inequalities for multiplicative Diophantine exponents.}
    Proc. Steklov Inst. Math., \textbf{275} (2011), 216--228.
\bibitem
    {cassels_DA}
    \textsc{J.\,W.\,S.\,Cassels}
    \textit{An introduction to Diophantine approximation.}
    Cambridge University Press (1957).
\bibitem
    {schmidt_DA}
    \textsc{W.\,M.\,Schmidt}
    \textit{Diophantine Approximation.}
    Lecture Notes in Math., \textbf{785}, Springer-Verlag (1980).
\bibitem
    {german_MJCNT_2012}
    \textsc{O.\,N.\,German}
    \textit{On Diophantine exponents and Khintchine's transference principle.}
    Moscow J. Comb. Number Theory, \textbf{2}:2 (2012), 22--51 
\bibitem
    {german_evdokimov}
    \textsc{O.\,N.\,German, K.\,G.\,Evdokimov}
    \textit{A strengthening of Mahler's transference theorem.}
    Izv. Math., \textbf{79}:1 (2015), 60--73.
\bibitem
    {mahler_casopis_linear}
    \textsc{K.\,Mahler}
    \textit{Ein Übertragungsprinzip für lineare Ungleichungen.}
    \v{C}as. Pe\v{s}t. Mat. Fys., \textbf{68} (1939), 85--92.
\bibitem
    {mahler_compound_bodies_I}
    \textsc{K.\,Mahler}
    \textit{On compound convex bodies, I.}
    Proc. London Math. Soc. (3), \textbf{5} (1955), 358--379.
\bibitem
    {mahler_compound_bodies_II}
    \textsc{K.\,Mahler}
    \textit{On compound convex bodies, II.}
    Proc. London Math. Soc. (3), \textbf{5} (1955), 380--384.




\end{thebibliography}
\end{document}